\documentclass[12pt,reqno]{amsart}
\usepackage{amsfonts}
\usepackage{amsmath}
\usepackage{amssymb,latexsym}
\usepackage{enumerate}
\usepackage[bookmarksnumbered, colorlinks, plainpages]{hyperref}
\usepackage{amsbsy}
\usepackage{amscd}
\usepackage{epsfig}
\usepackage{graphicx}
\usepackage{epstopdf}
\usepackage{caption}
\usepackage{subcaption}

\newtheorem{theorem}{Theorem}[section]
\newtheorem{lemma}{Lemma}[section]

\newtheorem{corollary}{Corollary}[section]
\newtheorem{definition}{Definition}[section]

\newtheorem{example}{Example}[section]

\newtheorem{case}{Case}
\numberwithin{equation}{section}

\begin{document}
\date{{\scriptsize Received: , Accepted: .}}
\title[A Geometric Interpretation to Fixed-Point Theory on $S_{b}$-Metric
Spaces]{A Geometric Interpretation to Fixed-Point Theory on $S_{b}$-Metric
Spaces}
\author[H. AYT\.{I}MUR]{H\"{u}lya AYT\.{I}MUR}
\address{Bal\i kesir University, Department of Mathematics, 10145 Bal\i
kesir, Turkey}
\email{hulya.aytimur@balikesir.edu.tr}
\author[N. TA\c{S}]{Nihal TA\c{S}}
\address{Bal\i kesir University, Department of Mathematics, 10145 Bal\i
kesir, Turkey}
\email{nihaltas@balikesir.edu.tr}
\maketitle

\begin{abstract}
In this paper we present some fixed-figure theorems as a geometric approach
to the fixed-point theory when the number of fixed points of a self-mapping
is more than one. To do this, we modify the Jleli-Samet type contraction and
define new contractions on $S_{b}$-metric spaces. Also, we give some
necessary examples to show the validity of our theoretical results.\newline
\end{abstract}




%

%

\section{\textbf{Introduction and Background}}

\label{sec:intro}

Classical fixed-point theory started with the Banach fixed-point theorem
\cite{Banach}. This theory is one of the useful tool of mathematical studies
and is an applicable area to topology, analysis, geometry, applied
mathematics, engineering etc. Metric fixed-point theory has been studied and
generalized with various aspects. One of these aspects is to generalize the
used contractive condition (for example, see \cite{Jleli-Samet}). Another
aspect is to generalize the used metric space such as, a $b$-metric space,
an $S$-metric space and an $S_{b}$-metric space as follows:

\begin{definition}
\label{def2} \cite{Bakhtin} Let $X$ be a nonempty set, $b\geq 1$ a given
real number and $d:X\times X\rightarrow \lbrack 0,\infty )$ a function
satisfying the following conditions for all $x,y,z\in X:$

$(b1)$ $d(x,y)=0$ if and only if $x=y$.

$(b2)$ $d(x,y)=d(y,x)$.

$(b3)$ $d(x,z)\leq b[d(x,y)+d(y,z)]$.

Then the function $d$ is called a $b$-metric on $X$ and the pair $(X,d)$ is
called a $b$-metric space.
\end{definition}

\begin{definition}
\cite{sedghi2} \label{def1} Let $X$ be a nonempty set and $S:X\times X\times
X\rightarrow \lbrack 0,\infty )$ be a function satisfying the following
conditions for all $x,y,z,a\in X:$

$(S1)$ $S(x,y,z)=0$ if and only if $x=y=z$,

$(S2)$ $S(x,y,z)\leq S(x,x,a)+S(y,y,a)+S(z,z,a)$.

Then $S$ is called an $S$-metric on $X$ and the pair $(X,S)$ is called an $S$%
-metric space.
\end{definition}

\begin{definition}
\cite{Sedghi-2016} \label{def3} Let $X$ be a nonempty set and $b\geq 1$ be a
given real number. A function $S_{b}:X\times X\times X\rightarrow \lbrack
0,\infty )$ is said to be $S_{b}$-metric if and only if for all $x,y,z,a\in
X $ the following conditions are satisfied:

$(S_{b}1)$ $S_{b}(x,y,z)=0$ if and only if $x=y=z$,

$(S_{b}2)$ $S_{b}(x,y,z)\leq b[S_{b}(x,x,a)+S_{b}(y,y,a)+S_{b}(z,z,a)]$.

The pair $(X,S_{b})$ is called an $S_{b}$-metric space.
\end{definition}

An $S_{b}$-metric space is also a generalization of an $S$-metric space
because every $S$-metric is an $S_{b}$-metric with $b=1$. But the converse
of this statement is not always true as seen in the following example.

\begin{example}
\cite{Tas-1} \label{exm1} Let $X=%
\mathbb{R}
$ and the function $S_{b}$ be defined by%
\begin{equation*}
S_{b}(x,y,z)=S(x,y,z)^{2}=\dfrac{1}{16}(\left\vert x-y\right\vert
+\left\vert y-z\right\vert +\left\vert x-z\right\vert )^{2}\text{,}
\end{equation*}%
for all $x,y,z\in
\mathbb{R}
$. Then the function $S_{b}$ is an $S_{b}$-metric with $b=4$, but it is not
an $S$-metric.
\end{example}

We see that the relationships between a $b$-metric and an $S_{b}$-metric as
follows:

\begin{lemma}
\cite{Tas-1} \label{lem_b_Sb} Let $(X,S_{b})$ be an $S_{b}$-metric space, $%
S_{b}$ be a symmetric $S_{b}$-metric with $b\geq 1$ and the function $%
d:X\times X\rightarrow \lbrack 0,\infty )$ be defined by%
\begin{equation*}
d(x,y)=S_{b}(x,x,y)\text{,}
\end{equation*}%
for all $x,y\in X$. Then $d$ is a $b$-metric on $X$.
\end{lemma}

\begin{lemma}
\cite{Tas-1} \label{lem_Sb_b} Let $(X,d)$ be a $b$-metric space with $b\geq
1 $ and the function $S_{b}:X\times X\times X\rightarrow \lbrack 0,\infty )$
be defined by%
\begin{equation*}
S_{b}(x,y,z)=d(x,z)+d(y,z)\text{,}
\end{equation*}%
for all $x,y,z\in X$. Then $S_{b}$ is an $S_{b}$-metric on $X$.
\end{lemma}

Recently, as a geometric generalization of a fixed-point theory,
fixed-circle problem has been studied. This problem was occurred in \cite%
{Ozgur-malaysian} and investigated some solutions to the this problem using
different approaches (for example, see \cite{Mlaiki arxiv}, \cite{Ozgur-Aip}%
, \cite{Ozgur-malaysian}, \cite{Ozgur-chapter}, \cite{Ozgur-simulation},
\cite{Pant-1}, \cite{Pant-2}, \cite{Pant-3}, \cite{Tas 2020} and the
references therein). Especially, this problem was studied on $S_{b}$-metric
space in \cite{Ozgur-chapter} and obtained some fixed-circle results using
the following basic definitions.

\begin{definition}
\cite{Ozgur-chapter} \label{defSb1} Let $(X,S_{b})$ be an $S_{b}$-metric
space with $b\geq 1$ and $x_{0}\in X$, $r\in \left( 0,\infty \right) $. The
circle centered at $x_{0}$ with radius $r$ is defined by%
\begin{equation*}
C_{x_{0},r}^{S_{b}}=\left\{ x\in X:S_{b}(x,x,x_{0})=r\right\} \text{.}
\end{equation*}
\end{definition}

\begin{definition}
\cite{Ozgur-chapter} \label{defSb2} Let $(X,S_{b})$ be an $S_{b}$-metric
space with $b\geq 1$, $C_{x_{0},r}^{S_{b}}$ be a circle on $X$ and $%
T:X\rightarrow X$ be a self-mapping. If $Tx=x$ for all $x\in
C_{x_{0},r}^{S_{b}}$ then the circle $C_{x_{0},r}^{S_{b}}$ is called as the
fixed circle of $T$.
\end{definition}

The notion of a fixed figure was defined as a generalization of the notions
of a fixed circle and a fixed disc as follows:

A geometric figure $\mathcal{F}$ (a circle, an ellipse, a hyperbola, a
Cassini curve etc.) contained in the fixed point set $Fix\left( T\right)
=\left\{ x\in X:x=Tx\right\} $ is called a\textit{\ fixed figure} (a fixed
circle, a fixed ellipse, a fixed hyperbola, a fixed Cassini curve, etc.) of
the self-mapping $T$ (see \cite{Ozgur-figure}). For this purpose, some
fixed-figure theorems were obtained using different aspects (see, \cite%
{Ercinar}, \cite{Joshi}, \cite{Ozgur-figure} and \cite{Tas-chapter} for more
details).

By the above motivation, the main of this paper is to obtain some
fixed-figure results on an $S_{b}$-metric space. To do this, we define new
Jleli-Samet type contractions. Using these new contractions, we prove
fixed-disc results, fixed-ellipse results, fixed-hyperbola results,
fixed-Cassini curve results and fixed-Apollonius circle results on an $S_{b}$%
-metric space.\ Also, we give an example to show the validity of our
obtained theorems.

\section{\textbf{Main Results}}

\label{sec:1} In this section, we present some fixed-figure results on an $%
S_{b}$-metric space. Before these results, we give the following definitions:

\begin{definition}
\label{df1} Let $(X,S_{b})$ be an $S_{b}$-metric space with $b\geq 1$ and $%
x_{0},x_{1},x_{2}\in X$, $r\in \lbrack 0,\infty )$.

\begin{enumerate}
\item The disc centered at $x_{0}$ with radius $r$ is defined by%
\begin{equation*}
D_{x_{0},r}^{S_{b}}=\left\{ x\in X:S_{b}(x,x,x_{0})\leq r\right\} \text{.}
\end{equation*}

\item The ellipse $E_{r}^{S_{b}}(x_{1},x_{2})$ is defined by
\begin{equation*}
E_{r}^{S_{b}}(x_{1},x_{2})=\left\{ x\in X:S_{b}\left( x,x,x_{1}\right)
+S_{b}\left( x,x,x_{2}\right) =r\right\} \text{.}
\end{equation*}

\item The hyperbola $H_{r}^{S_{b}}(x_{1},x_{2})$ is defined by
\begin{equation*}
H_{r}^{S_{b}}(x_{1},x_{2})=\left\{ x\in X:\left\vert S_{b}\left(
x,x,x_{1}\right) -S_{b}\left( x,x,x_{2}\right) \right\vert =r\right\} \text{.%
}
\end{equation*}

\item The Cassini curve $C_{r}^{S_{b}}(x_{1},x_{2})$ is defined by
\begin{equation*}
C_{r}^{S_{b}}(x_{1},x_{2})=\left\{ x\in X:S_{b}\left( x,x,x_{1}\right)
S_{b}\left( x,x,x_{2}\right) =r\right\} \text{.}
\end{equation*}

\item The Apollonius circle $A_{r}^{S_{b}}(x_{1},x_{2})$ is defined by
\begin{equation*}
A_{r}^{S_{b}}(x_{1},x_{2})=\left\{ x\in X-\{x_{2}\}:\frac{S_{b}\left(
x,x,x_{1}\right) }{S_{b}\left( x,x,x_{2}\right) }=r\right\} \text{.}
\end{equation*}
\end{enumerate}
\end{definition}

Now, we give the following example.

\begin{example}
\label{ex1} Let $\left( X,d\right) $ be a metric space and let us consider
the $S_{b}$-metric space $(X,S_{b})$ with the $S_{b}$-metric $S_{b}:X\times
X\times X\rightarrow \lbrack 0,\infty )$ defined as%
\begin{equation*}
S_{b}(x,y,z)=\left[ d(x,y)+d(y,z)+d(x,z)\right] ^{p}\text{,}
\end{equation*}%
for all $x,y,z\in X$ and $p>1$ \cite{Tas-1}. Let us consider $X=%
\mathbb{R}
^{3}$, the metric $d$ be a usual metric with $d(x,y)=\left\vert
x-y\right\vert $ and $p=3$. If we take $x_{0}=\left( 1,1,1\right) $ and $%
r=40 $, then we obtain the circle $C_{x_{0},r}^{S_{b}}$ as%
\begin{eqnarray*}
C_{x_{0},r}^{S_{b}} &=&\left\{ x\in
\mathbb{R}
^{3}:S_{b}(x,x,x_{0})=40\right\} \\
&=&\left\{ x\in
\mathbb{R}
^{3}:\left\vert x-1\right\vert ^{3}+\left\vert y-1\right\vert
^{3}+\left\vert z-1\right\vert ^{3}=5\right\}
\end{eqnarray*}%
and the disc $D_{x_{0},r}^{S_{b}}$ as%
\begin{eqnarray*}
D_{x_{0},r}^{S_{b}} &=&\left\{ x\in
\mathbb{R}
^{3}:S_{b}(x,x,x_{0})\leq 40\right\} \\
&=&\left\{ x\in
\mathbb{R}
^{3}:\left\vert x-1\right\vert ^{3}+\left\vert y-1\right\vert
^{3}+\left\vert z-1\right\vert ^{3}\leq 5\right\} \text{.}
\end{eqnarray*}

\begin{figure}[h]
\centering
\begin{subfigure}{.5\textwidth}
  \centering
  \includegraphics[width=.4\linewidth]{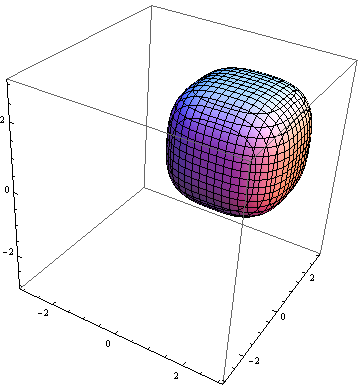}
  \caption{\Small The circle $C_{x_{0},r}^{S_{b}}$}
  \label{fig:1a}
\end{subfigure}%
\begin{subfigure}{.5\textwidth}
  \centering
  \includegraphics[width=.4\linewidth]{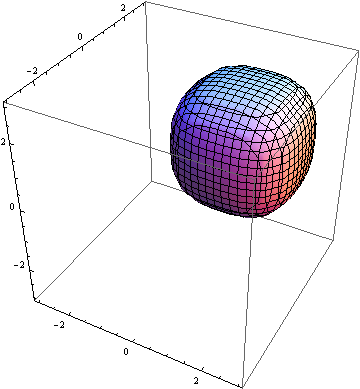}
  \caption{\Small  The disc $D_{x_{0},r}^{S_{b}}$}
  \label{fig:1b}
\end{subfigure}
\caption{\Small The geometric figures of the circle and the disc}
\label{fig:1}
\end{figure}

Also, if we take $x_{1}=\left( 1,1,1\right) $, $x_{1}=\left( -1,-1,-1\right)
$ and $r=400$, then we obtain the ellipse $E_{r}^{S_{b}}(x_{1},x_{2})$ as%
\begin{eqnarray*}
E_{r}^{S_{b}}(x_{1},x_{2}) &=&\left\{ x\in
\mathbb{R}
^{3}:S_{b}\left( x,x,x_{1}\right) +S_{b}\left( x,x,x_{2}\right) =400\right\}
\\
&=&\left\{
\begin{array}{c}
x\in
\mathbb{R}
^{3}:\left( \left\vert x-1\right\vert +\left\vert x+1\right\vert \right)
^{3}+\left( \left\vert y-1\right\vert +\left\vert y+1\right\vert \right) ^{3}
\\
+\left( \left\vert z-1\right\vert +\left\vert z+1\right\vert \right)
^{3}\leq 50%
\end{array}%
\right\} \text{.}
\end{eqnarray*}

\begin{figure}[h]
\centering
\includegraphics[width=.4\linewidth]{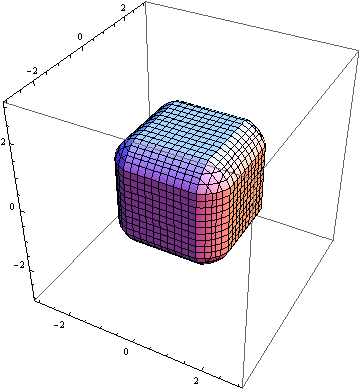}
\caption{\Small The ellipse $E_{r}^{S_{b}}(x_{1},x_{2})$}
\label{fig:2}
\end{figure}

If we take $x_{1}=\left( 1,1,1\right) $, $x_{1}=\left( -1,-1,-1\right) $ and
$r=40$, then we obtain the hyperbola $H_{r}^{S_{b}}(x_{1},x_{2})$ as%
\begin{eqnarray*}
H_{r}^{S_{b}}(x_{1},x_{2}) &=&\left\{ x\in
\mathbb{R}
^{3}:\left\vert S_{b}\left( x,x,x_{1}\right) -S_{b}\left( x,x,x_{2}\right)
\right\vert =40\right\} \\
&=&\left\{
\begin{array}{c}
x\in
\mathbb{R}
^{3}:\left\vert \left\vert x-1\right\vert -\left\vert x+1\right\vert
\right\vert ^{3}+\left\vert \left\vert y-1\right\vert -\left\vert
y+1\right\vert \right\vert ^{3} \\
+\left\vert \left\vert z-1\right\vert -\left\vert z+1\right\vert \right\vert
^{3}\leq 5%
\end{array}%
\right\} \text{,}
\end{eqnarray*}%
the Cassini curve $C_{r}^{S_{b}}(x_{1},x_{2})$ as%
\begin{eqnarray*}
C_{r}^{S_{b}}(x_{1},x_{2}) &=&\left\{ x\in
\mathbb{R}
^{3}:S_{b}\left( x,x,x_{1}\right) S_{b}\left( x,x,x_{2}\right) =40\right\} \\
&=&\left\{
\begin{array}{c}
x\in
\mathbb{R}
^{3}:\left( \left\vert x-1\right\vert \left\vert x+1\right\vert \right)
^{3}+\left( \left\vert y-1\right\vert \left\vert y+1\right\vert \right) ^{3}
\\
+\left( \left\vert z-1\right\vert \left\vert z+1\right\vert \right) ^{3}\leq
5%
\end{array}%
\right\}
\end{eqnarray*}%
and the Apollonius circle $A_{r}^{S_{b}}(x_{1},x_{2})$ as%
\begin{eqnarray*}
A_{r}^{S_{b}}(x_{1},x_{2}) &=&\left\{ x\in
\mathbb{R}
^{3}:\frac{S_{b}\left( x,x,x_{1}\right) }{S_{b}\left( x,x,x_{2}\right) }%
=40\right\} \\
&=&\left\{ x\in
\mathbb{R}
^{3}:\left( \frac{\left\vert x-1\right\vert }{\left\vert x+1\right\vert }%
\right) ^{3}+\left( \frac{\left\vert y-1\right\vert }{\left\vert
y+1\right\vert }\right) ^{3}+\left( \frac{\left\vert z-1\right\vert }{%
\left\vert z+1\right\vert }\right) ^{3}\leq 5\right\} \text{.}
\end{eqnarray*}

\begin{figure}[h]
\centering
\begin{subfigure}{.4\textwidth}
  \centering
  \includegraphics[width=.3\linewidth]{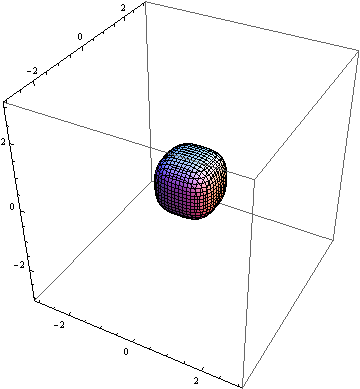}
  \caption{\Small The hyperbola $H_{r}^{S_{b}}(x_{1},x_{2})$}
  \label{fig:1A}
\end{subfigure}%
\begin{subfigure}{.5\textwidth}
  \centering
  \includegraphics[width=.3\linewidth]{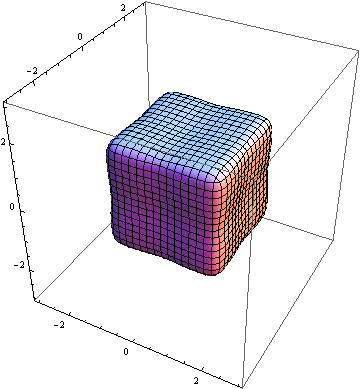}
  \caption{\Small  The Cassini curve $C_{r}^{S_{b}}(x_{1},x_{2})$}
  \label{fig:1B}
\end{subfigure}\newline
\begin{subfigure}{.5\textwidth}
  \centering
  \includegraphics[width=.3\linewidth]{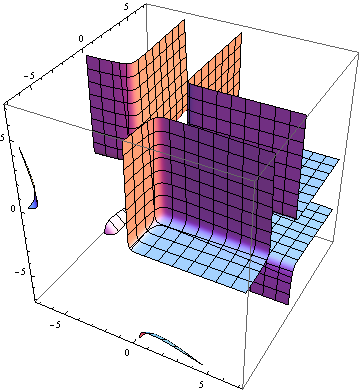}
  \caption{\Small  The Apollonius circle $A_{r}^{S_{b}}(x_{1},x_{2})$}
  \label{fig:1C}
\end{subfigure}
\caption{\Small The geometric figures of the hyperbola, Cassini curve and
Apollonius circle}
\label{fig:3}
\end{figure}
\end{example}

We give the following definitions of new notions to obtain some fixed-figure
results.

\begin{definition}
\label{df2} Let $(X,S_{b})$ be an $S_{b}$-metric space with $b\geq 1$ and $%
f:X\rightarrow X$ be a self-mapping. A geometric figure $\mathcal{F}$
contained in the fixed point set $Fix\left( f\right) $ is called a\textit{\
fixed figure} of the self-mapping $f$.
\end{definition}

\begin{definition}
\label{dfn2} Let $\left( X,S_{b}\right) $ be an $S_{b}$-metric space and $%
f:X\rightarrow X$ a self mapping. If there exists $x_{0}\in X$ such that%
\begin{equation*}
S_{b}\left( x,x,fx\right) >0\Rightarrow \varphi \left( S_{b}\left(
x,x,fx\right) \right) \leq \left[ \varphi \left( S_{b}\left(
x,x,x_{0}\right) \right) \right] ^{\alpha }
\end{equation*}

for all $x\in X$ where $\alpha \in \left( 0,1\right) $ and the function $%
\varphi :\left( 0,\infty \right) \rightarrow \left( 1,\infty \right) $ is
such that $\varphi $ is non-decreasing, then $f$ is called Jleli-Samet type $%
D_{x_{0}}$-$S_{b}$-contraction.
\end{definition}

\begin{theorem}
\label{thrm1} Let $\left( X,S_{b}\right) $ be an $S_{b}$-metric space and $%
f:X\rightarrow X$ Jleli-Samet type $D_{x_{0}}$-$S_{b}$-contraction with $%
x_{0}\in X$ and the number $r$ defined as
\begin{equation}
r=\inf \left\{ S_{b}\left( x,x,fx\right) :x\neq fx,x\in X\right\} .
\label{number1}
\end{equation}

Then $f$ fixes the disc $D_{x_{0},r}^{S_{b}}.$
\end{theorem}

\begin{proof}
At first, we show $fx_{0}=x_{0}$. On the contrary, let $fx_{0}\neq x_{0}.$
Using the Jleli-Samet type $D_{x_{0}}$-$S_{b}$-contraction hypothesis, we get%
\begin{eqnarray*}
\varphi \left( S_{b}\left( x_{0},x_{0},fx_{0}\right) \right)  &\leq &\left[
\varphi \left( S_{b}\left( x_{0},x_{0},x_{0}\right) \right) \right] ^{\alpha
} \\
&=&\left[ \varphi \left( 0\right) \right] ^{\alpha },
\end{eqnarray*}

a contradiction. So we get
\begin{equation}
fx_{0}=x_{0}.  \label{number2}
\end{equation}

To show that $f$ fixes the disc $D_{x_{0},r}^{S_{b}}$, we consider the
following cases:

\begin{case}
\label{cs1} Let $r=0.$ Then we have $D_{x_{0},r}^{S_{b}}=\left\{
x_{0}\right\} $ and by the equality (\ref{number2}), we get $fx_{0}=x_{0}.$
\end{case}

\begin{case}
\label{cs2} Let $r>0$ and $x\in D_{x_{0},r}^{S_{b}}$ be any point such that $%
x\neq fx.$ Using the hypothesis, we obtain%
\begin{eqnarray*}
\varphi \left( S_{b}\left( x,x,fx\right) \right)  &\leq &\left[ \varphi
\left( S_{b}\left( x,x,x_{0}\right) \right) \right] ^{\alpha } \\
&\leq &\left[ \varphi \left( r\right) \right] ^{\alpha } \\
&\leq &\left[ \varphi \left( S_{b}\left( x,x,fx\right) \right) \right]
^{\alpha }
\end{eqnarray*}

a contradiction with $\alpha \in \left( 0,1\right) .$ Hence, it should be $%
fx=x.$ Consequently $f$ fixes the disc $D_{x_{0},r}^{S_{b}}$ .
\end{case}
\end{proof}

Now we give the following corollary:

\begin{corollary}
\label{col1} If we take $b=1,$ then we get Theorem 2.2 in \cite{Tas-full}.
\end{corollary}

\begin{definition}
\label{dfn3} Let $\left( X,S_{b}\right) $ be an $S_{b}$-metric space and $%
f:X\rightarrow X$ a self mapping. If there exists $x_{1},x_{2}\in X$ such
that%
\begin{equation*}
S_{b}\left( x,x,fx\right) >0\Rightarrow \varphi \left( S_{b}\left(
x,x,fx\right) \right) \leq \left[ \varphi \left( S_{b}\left(
x,x,x_{1}\right) +S_{b}\left( x,x,x_{2}\right) \right) \right] ^{\alpha }
\end{equation*}

for all $x\in X\backslash \left\{ x_{1},x_{2}\right\} $ where $\alpha \in
\left( 0,1\right) $ and the function $\varphi :\left( 0,\infty \right)
\rightarrow \left( 1,\infty \right) $ is such that $\varphi $ is
non-decreasing, then $f$ is called Jleli-Samet type $E_{x_{1},x_{2}}$-$S_{b}$%
-contraction.
\end{definition}

\begin{theorem}
\label{thrm2} Let $\left( X,S_{b}\right) $ be an $S_{b}$-metric space and $%
f:X\rightarrow X$ Jleli-Samet type $E_{x_{1},x_{2}}$-$S_{b}$-contraction
with $x_{1},x_{2}\in X$ and the number $r$ defined as (\ref{number1}). If \ $%
fx_{1}=x_{1}$ and $fx_{2}=x_{2},$ then $f$ fixes the ellipse $%
E_{r}^{S_{b}}\left( x_{1},x_{2}\right) .$
\end{theorem}

\begin{proof}
We consider the following cases:

\begin{case}
\label{cs3} Let $r=0$. Then we have $x_{1}=x_{2}$ and $E_{r}^{S_{b}}\left(
x_{1},x_{2}\right) =\left\{ x_{1}\right\} =\left\{ x_{2}\right\} .$ Using
the hypothesis we have $fx_{1}=x_{1}$ and $fx_{2}=x_{2}.$
\end{case}

\begin{case}
\label{cs4} Let $r>0$ and $x\in E_{r}^{S_{b}}\left( x_{1},x_{2}\right) $ be
any point such that $x\neq fx.$ Using the hypothesis we get%
\begin{eqnarray*}
\varphi \left( S_{b}\left( x,x,fx\right) \right) &\leq &\left[ \varphi
\left( S_{b}\left( x,x,x_{1}\right) +S_{b}\left( x,x,x_{2}\right) \right) %
\right] ^{\alpha } \\
&\leq &\left[ \varphi \left( r\right) \right] ^{\alpha } \\
&\leq &\left[ \varphi \left( S_{b}\left( x,x,fx\right) \right) \right]
^{\alpha }
\end{eqnarray*}

a contradiction with $\alpha \in \left( 0,1\right) .$ Hence it should be $%
fx=x.$ Consequently $f$ \ fixes the ellipse $E_{r}^{S_{b}}\left(
x_{1},x_{2}\right) .$
\end{case}
\end{proof}

\begin{corollary}
\label{col2} If we take $b=1,$ then we get fixed ellipse results on an $S$%
-metric space.
\end{corollary}

\begin{definition}
\label{dfn4} Let $\left( X,S_{b}\right) $ be an $S_{b}$-metric space and $%
f:X\rightarrow X$ a self mapping. If there exists $x_{1},x_{2}\in X$ such
that%
\begin{equation*}
S_{b}\left( x,x,fx\right) >0\Rightarrow \varphi \left( S_{b}\left(
x,x,fx\right) \right) \leq \left[ \varphi \left( \left\vert S_{b}\left(
x,x,x_{1}\right) -S_{b}\left( x,x,x_{2}\right) \right\vert \right) \right]
^{\alpha }
\end{equation*}

for all $x\in X\backslash \left\{ x_{1},x_{2}\right\} $ where $\alpha \in
\left( 0,1\right) $ and the function $\varphi :\left( 0,\infty \right)
\rightarrow \left( 1,\infty \right) $ is such that $\varphi $ is
non-decreasing, then $f$ is called Jleli-Samet type $H_{x_{1},x_{2}}$-$S_{b}$%
-contraction.
\end{definition}

\begin{theorem}
\label{thrm3} Let $\left( X,S_{b}\right) $ be an $S_{b}$-metric space and $%
f:X\rightarrow X$ Jleli-Samet type $H_{x_{1},x_{2}}$-$S_{b}$-contraction
with $x_{1},x_{2}\in X$ and the number $r$ defined as (\ref{number1}). If \ $%
fx_{1}=x_{1}$ and $fx_{2}=x_{2}$ and $r>0,$ then $f$ fixes the hyperbola $%
H_{r}^{S_{b}}\left( x_{1},x_{2}\right) .$
\end{theorem}

\begin{proof}
Let $x\in H_{r}^{S_{b}}\left( x_{1},x_{2}\right) $ be any point such that $%
x\neq fx.$ Using the hypothesis we get%
\begin{eqnarray*}
\varphi \left( S_{b}\left( x,x,fx\right) \right) &\leq &\left[ \varphi
\left( \left\vert S_{b}\left( x,x,x_{1}\right) -S_{b}\left( x,x,x_{2}\right)
\right\vert \right) \right] ^{\alpha } \\
&\leq &\left[ \varphi \left( r\right) \right] ^{\alpha } \\
&\leq &\left[ \varphi \left( S_{b}\left( x,x,fx\right) \right) \right]
^{\alpha }
\end{eqnarray*}

a contradiction with $\alpha \in \left( 0,1\right) .$ Hence it should be $%
fx=x.$ Consequently $f$ \ fixes the hyperbola $H_{r}^{S_{b}}\left(
x_{1},x_{2}\right) .$
\end{proof}

\begin{corollary}
\label{col3} If we take $b=1,$ then we get fixed hyperbola results on an $S$%
-metric space
\end{corollary}

\begin{definition}
\label{dfn5} Let $\left( X,S_{b}\right) $ be an $S_{b}$-metric space and $%
f:X\rightarrow X$ a self mapping. If there exists $x_{1},x_{2}\in X$ such
that%
\begin{equation*}
S_{b}\left( x,x,fx\right) >0\Rightarrow \varphi \left( S_{b}\left(
x,x,fx\right) \right) \leq \left[ \varphi \left( S_{b}\left(
x,x,x_{1}\right) S_{b}\left( x,x,x_{2}\right) \right) \right] ^{\alpha }
\end{equation*}

for all $x\in X\backslash \left\{ x_{1},x_{2}\right\} $ where $\alpha \in
\left( 0,1\right) $ and the function $\varphi :\left( 0,\infty \right)
\rightarrow \left( 1,\infty \right) $ is such that $\varphi $ is
non-decreasing, then $f$ is called Jleli-Samet type $C_{x_{1},x_{2}}$-$S_{b}$%
-contraction.
\end{definition}

\begin{theorem}
\label{thrm4} Let $\left( X,S_{b}\right) $ be an $S_{b}$-metric space and $%
f:X\rightarrow X$ Jleli-Samet type $C_{x_{1},x_{2}}$-$S_{b}$-contraction
with $x_{1},x_{2}\in X$ and the number $r$ defined as (\ref{number1}). If \ $%
fx_{1}=x_{1}$ and $fx_{2}=x_{2},$ then $f$ fixes the Cassini curve $%
C_{r}^{S_{b}}\left( x_{1},x_{2}\right) .$
\end{theorem}

\begin{proof}
We consider the following cases:

\begin{case}
\label{cs5} Let $r=0$. Then we have $x_{1}=x_{2}$ and $C_{r}^{S_{b}}\left(
x_{1},x_{2}\right) =\left\{ x_{1}\right\} =\left\{ x_{2}\right\} .$ Using
the hypothesis we have $fx_{1}=x_{1}$ and $fx_{2}=x_{2}.$
\end{case}

\begin{case}
\label{cs6} Let $r>0$ and $x\in C_{r}^{S_{b}}\left( x_{1},x_{2}\right) $ be
any point such that $x\neq fx.$ Using the hypothesis we get%
\begin{eqnarray*}
\varphi \left( S_{b}\left( x,x,fx\right) \right) &\leq &\left[ \varphi
\left( S_{b}\left( x,x,x_{1}\right) S_{b}\left( x,x,x_{2}\right) \right) %
\right] ^{\alpha } \\
&\leq &\left[ \varphi \left( r\right) \right] ^{\alpha } \\
&\leq &\left[ \varphi \left( S_{b}\left( x,x,fx\right) \right) \right]
^{\alpha }
\end{eqnarray*}

a contradiction with $\alpha \in \left( 0,1\right) .$ Hence it should be $%
fx=x.$ Consequently $f$ \ fixes the Cassini curve $C_{r}^{S_{b}}\left(
x_{1},x_{2}\right) .$
\end{case}
\end{proof}

\begin{corollary}
\label{col4} If we take $b=1,$ then we get fixed Cassini curve results on an
$S$-metric space.
\end{corollary}

\begin{definition}
\label{dfn6} Let $\left( X,S_{b}\right) $ be an $S_{b}$-metric space and $%
f:X\rightarrow X$ a self mapping. If there exists $x_{1},x_{2}\in X$ such
that%
\begin{equation*}
S_{b}\left( x,x,fx\right) >0\Rightarrow \varphi \left( S_{b}\left(
x,x,fx\right) \right) \leq \left[ \varphi \left( \frac{S_{b}\left(
x,x,x_{1}\right) }{S_{b}\left( x,x,x_{2}\right) }\right) \right] ^{\alpha }
\end{equation*}

for all $x\in X\backslash \left\{ x_{1},x_{2}\right\} $ where $\alpha \in
\left( 0,1\right) $ and the function $\varphi :\left( 0,\infty \right)
\rightarrow \left( 1,\infty \right) $ is such that $\varphi $ is
non-decreasing, then $f$ is called Jleli-Samet type $A_{x_{1},x_{2}}$-$S_{b}$%
-contraction.
\end{definition}

\begin{theorem}
\label{thrm5} Let $\left( X,S_{b}\right) $ be an $S_{b}$-metric space and $%
f:X\rightarrow X$ Jleli-Samet type $A_{x_{1},x_{2}}$-$S_{b}$-contraction
with $x_{1},x_{2}\in X$ and the number $r$ defined as (\ref{number1}). If \ $%
fx_{1}=x_{1}$ and $fx_{2}=x_{2},$ then $f$ fixes the Apollonius circle $%
A_{r}^{S_{b}}\left( x_{1},x_{2}\right) .$
\end{theorem}

\begin{proof}
We consider the following cases:

\begin{case}
\label{cs7} Let $r=0$. Then we have $x_{1}=x_{2}$ and $A_{r}^{S_{b}}\left(
x_{1},x_{2}\right) =\left\{ x_{1}\right\} =\left\{ x_{2}\right\} .$ Using
the hypothesis we have $fx_{1}=x_{1}$ and $fx_{2}=x_{2}.$
\end{case}

\begin{case}
\label{cs8} Let $r>0$ and $x\in A_{r}^{S_{b}}\left( x_{1},x_{2}\right) $ be
any point such that $x\neq fx.$ Using the hypothesis we get%
\begin{eqnarray*}
\varphi \left( S_{b}\left( x,x,fx\right) \right) &\leq &\left[ \varphi
\left( \frac{S_{b}\left( x,x,x_{1}\right) }{S_{b}\left( x,x,x_{2}\right) }%
\right) \right] ^{\alpha } \\
&\leq &\left[ \varphi \left( r\right) \right] ^{\alpha } \\
&\leq &\left[ \varphi \left( S_{b}\left( x,x,fx\right) \right) \right]
^{\alpha }
\end{eqnarray*}

a contradiction with $\alpha \in \left( 0,1\right) .$ Hence it should be $%
fx=x.$ Consequently $f$ \ fixes the Apollonius circle $C_{r}^{S_{b}}\left(
x_{1},x_{2}\right) .$
\end{case}
\end{proof}

\begin{corollary}
\label{col5} If we take $b=1,$ then we get fixed Apollonius circle results
on an $S$-metric space.
\end{corollary}

Finally we give the following illustrative example.

\begin{example}
\label{exm2} Let $X=\left[ -1,1\right] \cup \left\{ -7,-\sqrt{2},\sqrt{2},%
\frac{7}{3},7,8,21\right\} $ and the $S$-metric defined as%
\begin{equation*}
S(x,y,z)=\left\vert x-z\right\vert +\left\vert x+z-2y\right\vert \text{,}
\end{equation*}%
for all $x,y,z\in
\mathbb{R}
$ \cite{Ozgur-mathsci-Rhoades}. This $S$-metric is also an $S_{b}$-metric
with $b=1$. Let us define the function $f:X\rightarrow X$ as%
\begin{equation*}
fx=\left\{
\begin{array}{ccc}
x & , & X-\{8\} \\
7 & , & x=8%
\end{array}%
\right. \text{,}
\end{equation*}%
for all $x\in X$ and the function $\varphi :\left( 0,\infty \right)
\rightarrow \left( 1,\infty \right) $ as%
\begin{equation*}
\varphi (t)=t+1\text{,}
\end{equation*}%
for all $t>0$ with $r=2.$ Then,

$\triangleright $ The function $f$ is Jleli-Samet type $D_{x_{0}}$-$S_{b}$%
-contraction with $\alpha =0.5,x_{0}=0$. Consequently, $f$ fixes the disc $%
D_{0,2}^{S_{b}}=\left[ -1,1\right] .$

$\triangleright $ The function $f$ is Jleli-Samet type $E_{x_{1},x_{2}}$-$%
S_{b}$-contraction with $x_{1}=-\frac{1}{2}$, $x_{2}=\frac{1}{2}$ and $%
\alpha =0.5.$ Consequently, $f$ fixes the ellipse $E_{2}^{S_{b}}\left( -%
\frac{1}{2},\frac{1}{2}\right) =\left[ -\frac{1}{2},\frac{1}{2}\right] .$

$\triangleright $ The function $f$ is Jleli-Samet type $H_{x_{1},x_{2}}$-$%
S_{b}$-contraction with $x_{1}=-1,x_{2}=1$ and $\alpha =0.9.$ Consequently, $%
f$ fixes the hyperbola $H_{2}^{S_{b}}\left( -1,1\right) =\left\{ -\frac{1}{2}%
,\frac{1}{2}\right\} .$

$\triangleright $ The function $f$ is Jleli-Samet type $C_{x_{1},x_{2}}$-$%
S_{b}$-contraction with $x_{1}=-1,x_{2}=1$ and $\alpha =0.5.$ Consequently, $%
f$ fixes the Cassini curve $C_{2}^{S_{b}}\left( -1,1\right) =\left\{ -\sqrt{2%
},0,\sqrt{2}\right\} .$

$\triangleright $ The function $f$ is Jleli-Samet type $A_{x_{1},x_{2}}$-$%
S_{b}$-contraction with $x_{1}=-7,x_{2}=7$ and $\alpha =0.5.$ Consequently, $%
f$ fixes the Apollonius circle $A_{2}^{S_{b}}\left( -7,7\right) =\left\{
\frac{7}{3},21\right\} .$
\end{example}

\section{\textbf{Conclusion}}

In this paper, we present some new contractions and some fixed-figure
results on an $S_{b}$-metric space. The obtained results can be considered
as some geometric consequences of fixed-point theory. Using these
approaches, new geometric generalizations of known fixed-point theorems can
be studied on metric and generalized metric spaces.

\end{document}